\documentclass[reqno]{amsart}
\usepackage{amsmath, amsfonts,amssymb,amsthm,amscd,latexsym}
\usepackage{cite}
\usepackage{array}
\usepackage{mathrsfs}
\usepackage{color}
 \usepackage{lmodern}
\newtheorem{thm}{Theorem}[section]
\newtheorem{theorem}[thm]{Theorem}

\newtheorem{lemma}[thm]{Lemma}
\newtheorem{cor}[thm]{Corollary}
\newtheorem{corollary}[thm]{Corollary}
\newtheorem{proposition}[thm]{Proposition}

\newtheorem{definition}[thm]{Definition}
\newtheorem{remark}[thm]{Remark}

\newtheorem{question}{Question}

\newcommand{\cA}{{\mathcal A}}

\newcommand{\cE}{{\mathcal E}}
\newcommand{\cF}{{\mathcal F}}

\newcommand{\cH}{{\mathcal H}}
\newcommand{\cI}{{\mathcal I}}
\newcommand{\cJ}{{\mathcal J}}

\newcommand{\cM}{{\mathcal M}}
\newcommand{\cN}{{\mathcal N}}

\newcommand{\cP}{{\mathcal P}}

 \usepackage{color}
 \newcommand{\norm}[1]{\left\lVert#1\right\rVert}
\usepackage{hyperref}                   
\hypersetup{colorlinks,
    linkcolor=blue,%
    citecolor=blue}
\usepackage{etoolbox}

\usepackage{caption}
 \usepackage{float}

\begin{document}

 \title[The ball-covering property]{The ball-covering property of von Neumann algebras and noncommutative symmetric spaces}

\author[J. Huang]{Jinghao Huang}

 \author[K. Kudaybergenov]{Karimbergen Kudaybergenov}
\address{Institute for  Advanced Study in  Mathematics of HIT, Harbin Institute of Technology, Harbin, 150001, China}
\email{{\color{blue}jinghao.huang@hit.edu.cn}}
\email{{\color{blue}karim2006@mail.ru}}

 \author[R. Liu]{Rui Liu}
 \address{School of Mathematical Sciences and LPMC, Nankai University, Tianjin 300071, China}
 \email{{\color{blue}ruiliu@nankai.edu.cn}}

 \thanks{J. Huang was supported the NNSF of China (No.12031004 and 12301160). R. Liu was supported by the NNSF of China (No. 11671214, 11971348 and 12071230). }



\begin{abstract}
We fully characterize those  von Neumann algebras  having the ball-covering property. We also study the ball-covering property of noncommutative symmetric spaces. In particular, we  provide  a number of new  examples of non-separable (commutative and noncommutative) Banach spaces having the ball-covering property.
\end{abstract}

\maketitle

\section{Introduction}

A normed  space $X$ is said to have the {\it ball-covering property} (BCP, for short) if its unit sphere can be covered by the union of countably many open/closed balls not containing the origin\cite{Cheng}.{\
The centers of the balls are called the {\it BCP points} of~$X$.
It is worth noting that the sequence of BCP points is not necessarily bounded. 
Suppose that $\{B_n\}_{n\ge 1}$ is a sequence of balls off the origin whose union covers the unit sphere of $X$.
Then $X$ is said to have the {\it strong ball-covering property} (SBCP)~\cite{LZ2021}, if the radii of $\{B_n\}_{n\ge 1}$ is a bounded sequence. The space $X$ is said to have the {\it uniform ball-covering property} (UBCP) if $X$ has the SBCP, and there exists an open ball $B_0$ containing the origin such that $B_n \cap B_0 =\empty $ for all $n\in \mathbb{N}$.

It turns out that the  BCP is a powerful notion to study geometric and topological properties of Banach spaces, see e.g. \cite{CKZ,CLL,FZ,Shang,AMZ}.
It is known that separable Banach spaces have the BCP but the converse is not true \cite{Cheng, CCL}.
Actually, the BCP only implies the $w^*$-separability of the dual space~\cite{Cheng}.
On the other hand, every Banach space with a $w^*$-separable dual can be $(1+\varepsilon)$-renormed to have the BCP for any $\varepsilon>0$~\cite{CSZ,FZ}.

Cheng~\cite{Cheng}
obtained the first example of
  a non-separable space with the BCP by showing that    $\ell_\infty$ has such a property. However,   Cheng,   Cheng and  Liu \cite{CCL} showed the $\ell_\infty$ can be renormed to fail the BCP, which implies that BCP is not preserved by linear isomorphisms.
  Luo and Zheng showed that for a sequence of normed spaces,
  $$\left(\sum_{k=1}^\infty  \oplus X_k \right)_{\ell_\infty} $$
  has the BCP if and only if each Banach space $X_k$ has the BCP\cite{LZ2020}.
  Note that even though $L_\infty(0,1)$ is isomorphic to $\ell_\infty$\cite[Theorem 4.3.10]{AK},  
 $L_\infty (0,1)$ fails the BCP~\cite[Theorem 3.5]{LZ2020}.
There exist locally compact Hausdorff space $K$  such that $\ell_\infty =C_0(K_1)$ (respectively, $L_\infty(0,1)=C_0(K_2)$), 
the Banach space of all continuous functions on a locally compact Hausdorff space $K$  vanishing at infinity endowed with the supremum norm\cite[Theorem 4.4.3]{KR} (see also \cite[Theorem II.2.2.4]{Blackadar}). 
In a recent joint  paper of the third author\cite{LLLZ}, 
it was shown that $C_0(K)$ has the BCP/UBCP if and only if $K$ has a countable $\pi$-basis\cite[Theorem 2.1]{LLLZ}.
In particular, the locally compact Hausdorff space $K_1 $  
has a $\pi$-basis but the space $K_2$   fails to  have a $\pi$-basis.
Moreover,     a noncommutative version of  Cheng's result was  established  by 
  showing the $B(\ell_p)$, $1\le p<\infty$, has the UBCP\cite[Corollaries 3.1 and  3.4]{LLLZ}.

The first problem considered in this paper is  motivated by results obtained in~\cite{Cheng,LLLZ,LZ2020,CCL}.
\begin{question}
How to characterize those von Neumann algebras having the BCP?
  \end{question} In  Theorem~\ref{main1} below, 
  we show that a von Neumann algebra has the BCP (indeed, the UBCP) if and only if it is atomic and can be represented on a separable Hilbert space.
Our approach relies on the study of positive linear functionals   on von Neumann algebras,  
which is completely different from that of the commutative case  \cite{Cheng,LZ2020}.

Although the Lebesgue spaces $L_p$ ($1\le p\le \infty$) play a primary role in many areas of mathematical analysis, there are other classes of Banach spaces of measurable functions that are also of interest.
The larger classes of  Orlicz spaces and Lorentz spaces (which are not necessarily separable), for example, are of intrinsic importance.
There is a considerable literature dealing with each of these classes, see e.g. \cite{Bennett_S,LT2,KPS}.
In \cite[Theorems 2.5 and  2.6]{LZ2020}, it is shown that for a separable Lorentz sequence space $E=d(w,p)$, $1\le p<\infty $ (or a separable Orlicz sequence space  with the $\Delta_2$-condition) and a sequence of normed spaces,
the space $$\left(\sum\oplus X_k\right)_E$$ has the BCP if and only if all $X_k$'s have the BCP.
However, there are very few concrete examples of non-separable Banach spaces having the BCP.
 Another object of  this paper is to obtain a wide class of (non-separable) Banach spaces having the BCP.
We concentrate on the so-called symmetric function spaces \cite{Bennett_S,LT2} and their noncommutative counterpart~\cite{DP2014,
DPS, Kalton_S}, which is more general than Orlicz spaces and Lorentz spaces.

  \begin{question}
    For which (non-separable) symmetric function spaces $E(0,\infty)$ and semifinite von Neumann algebra $\cM$ (equipped with semifinite faithful normal traces~$\tau$), the noncommutative symmetric spaces $E(\cM,\tau)$ have the BCP?
  \end{question}

  In Theorem \ref{main}, we show that the Fatou norm is a sharp condition for $E(\cM,\tau)$ (on a separable Hilbert space) to have the UBCP/BCP, which provides a wide class of non-separable Banach spaces (e.g. Orlicz/Lorentz spaces and weak $L_p$-spaces, $p>1$) having the BCP.
   This result is new even for classical symmetric function spaces~\cite{Bennett_S,KPS, LT2}.
   On the other hand, we show that if $\cM$ can not be represented on a separable Hilbert space, then any fully symmetric space $E(\cM,\tau)$ fails the BCP, see Propositions \ref{prop:noncou} and  \ref{prop:nonseparable} below.

\section{Preliminaries}\label{s:p}

In this section, we recall main notions of the theory of noncommutative integration, introduce some properties of generalized singular value functions.
For details on von Neumann algebra
theory, the reader is referred to e.g. \cite{Blackadar}, \cite{Dixmier}, \cite{KR}
or \cite{Tak}. General facts concerning measurable operators may
be found in \cite{Nelson}, \cite{Se,Tak2} (see also the forthcoming book \cite{DPS}).
For convenience of the reader, some of the basic definitions are recalled.

In what follows,  $\cH$ is a (not necessarily separable) Hilbert space and $\left(B(\cH),\norm{\cdot}_\infty  \right)$ is the
$*$-algebra of all bounded linear operators on $\cH$, and
$\mathbf{1}$ is the identity operator on $\cH$.
Let $\mathcal{M}$ be
a von Neumann algebra on $\cH$.
Let   $\cP\left(\mathcal{M}\right)$  be the lattice of all
projections of $\mathcal{M}$.
We denote  by $\cM_p$ the reduced von Neumann algebra $p\cM p$ generated by a projection $p\in \cP(\cM)$.

\label{s:p1}
A linear operator $x:\mathfrak{D}\left( x\right) \rightarrow \cH $,
where the domain $\mathfrak{D}\left( x\right) $ of $x$ is a linear
subspace of $\cH$, is said to be {\it affiliated} with $\mathcal{M}$
if $yx \subseteq xy$ for all $y\in \mathcal{M}^{\prime }$, where $\mathcal{M}^{\prime }$ is the commutant of $\mathcal{M}$.
A linear
operator $x:\mathfrak{D}\left( x\right) \rightarrow \cH $ is termed
{\it measurable} with respect to $\mathcal{M}$ if $x$ is closed,
densely defined, affiliated with $\mathcal{M}$ and there exists a
sequence $\left\{ p_n\right\}_{n=1}^{\infty}$ in the set  $\cP\left(\mathcal{M}\right)$ of all
projections of $\mathcal{M}$ such
that $p_n\uparrow \mathbf{1}$, $p_n(\cH)\subseteq\mathfrak{D}\left(x \right) $
and $\mathbf{1}-p_n$ is a finite projection (with respect to $\mathcal{M}$)
for all $n$.
It should be noted that the condition $p _{n}\left(
\cH\right) \subseteq \mathfrak{D}\left( x\right) $ implies that
$xp _{n}\in \mathcal{M}$. The collection of all measurable
operators with respect to $\mathcal{M}$ is denoted by $S\left(
\mathcal{M} \right) $, which is a unital $\ast $-algebra
with respect to strong sums and products (denoted simply by $x+y$ and $xy$ for all $x,y\in S\left( \mathcal{M}\right) $).

Let $x$ be a self-adjoint operator affiliated with $\mathcal{M}$.
We denote its spectral measure by $\{e^x\}$.
It is well known that if
$x$ is a closed operator affiliated with $\mathcal{M}$ with the
polar decomposition $x = u|x|$, then $u\in\mathcal{M}$ and $e\in
\mathcal{M}$ for all projections $e\in \{e^{|x|}\}$.
Moreover,
$x \in S(\mathcal{M})$ if and only if $x $ is closed,
 densely
defined, affiliated with $\mathcal{M}$ and $e^{|x|}(\lambda,
\infty)$ is a finite projection for some $\lambda> 0$.
 It follows
immediately that in the case when $\mathcal{M}$ is a von Neumann
algebra of type $III$ or a type $I$ factor, we have
$S(\mathcal{M})= \mathcal{M}$.
For type $II$ von Neumann algebras,
this is no longer true.
From now on, let $\mathcal{M}$ be a
semifinite von Neumann algebra equipped with a faithful normal
semifinite trace $\tau$.

For any closed and densely defined linear operator $x :\mathfrak{D}\left( x \right) \rightarrow \cH $,
the \emph{null projection} $n(x)=n(|x|)$ is the projection onto its kernel $\mbox{Ker} (x)$.
The \emph{left support} $l(x )$ is the projection onto the closure of its range $\mbox{Ran}(x)$ and the \emph{right support} $r(x)$ of $x$ is defined by $r(x) ={\bf{1}} - n(x)$.

An operator $x\in S\left( \mathcal{M}\right) $ is called $\tau$-measurable if there exists a sequence
$\left\{p_n\right\}_{n=1}^{\infty}$ in $\cP \left(\mathcal{M}\right)$ such that
$p_n\uparrow \mathbf{1}$, $p_n\left( \cH \right)\subseteq \mathfrak{D}\left(x \right)$ and
$\tau(\mathbf{1}-p_n)<\infty $ for all $n$.
The collection of all $\tau $-measurable
operators is a unital $\ast $-subalgebra of $S\left(
\mathcal{M}\right) $,  denoted by $S\left( \mathcal{M}, \tau\right)
$.
It is well known that a linear operator $x$ belongs to $S\left(
\mathcal{M}, \tau\right) $ if and only if $x\in S(\mathcal{M})$
and there exists $\lambda>0$ such that $\tau(e^{|x|}(\lambda,
\infty))<\infty$.
Alternatively, an unbounded operator $x$
affiliated with $\mathcal{M}$ is  $\tau$-measurable (see
\cite{FK}) if and only if
$$\tau\left(e^{|x|}(n,\infty)\right)\rightarrow 0,\quad n\to\infty.$$

\begin{definition}
Let a semifinite von Neumann  algebra $\mathcal M$ be equipped
with a faithful normal semi-finite trace $\tau$ and let $x\in
S(\mathcal{M},\tau)$. The generalised singular value function $\mu(x):t\rightarrow \mu(t;x)$, $t>0$,  of
the operator $x$ is defined by setting
$$
\mu(t;x)
=
\inf \left\{
\left\|xp\right\|_\infty :\ p\in \cP(\cM), \ \tau(\mathbf{1}-p)\leq t
\right\}.
$$
\end{definition}
An equivalent definition in terms of the
distribution function of the operator $x$ is the following. For every self-adjoint
operator $x\in S(\mathcal{M},\tau) $,  setting
$$d_x(t)=\tau(e^{x}(t,\infty)),\quad t>0,$$
we have (see e.g. \cite{FK} and \cite{LSZ})
$$
\mu(t; x)=\inf\{s\geq0:\ d_{|x|}(s)\leq t\}.
$$
Note that $d_x(\cdot)$ is a right-continuous function (see e.g.  \cite{FK} and \cite{DPS}).

Consider the algebra $\mathcal{M}=L^\infty(0,\infty)$ of all
Lebesgue measurable essentially bounded functions on $(0,\infty)$.
Algebra $\mathcal{M}$ can be viewed as an abelian von Neumann
algebra acting via multiplication on the Hilbert space
$\mathcal{H}=L^2(0,\infty)$, with the trace given by integration
with respect to Lebesgue measure $m.$
It is easy to see that the
algebra of all $\tau$-measurable operators
affiliated with $\mathcal{M}$ can be identified with
the subalgebra $S(0,\infty)$ of the algebra of Lebesgue measurable functions which consists of all functions $f$ such that
$m(\{|f|>s\})$ is finite for some $s>0$. It should also be pointed out that the
generalised singular value function $\mu(f)$ is precisely the
decreasing rearrangement $\mu(f)$ of the function $|f|$ (see e.g. \cite{KPS,Bennett_S
}) defined by
$$\mu(t;f)=\inf\{s\geq0:\ m(\{|f|\geq s\})\leq t\}.$$

For convenience of the reader,  we also recall the definition of the \emph{measure topology} $t_\tau$ on the algebra $S(\cM,\tau)$. For every $\varepsilon,\delta>0,$ we define the set
$$V(\varepsilon,\delta)=\{x\in S(\mathcal{M},\tau):\ \exists p \in \cP\left(\mathcal{M}\right)\mbox{ such that }
\left\|x(\mathbf{1}-p)\right\|_\infty \leq\varepsilon,\ \tau(p)\leq\delta\}.$$ The topology
generated by the sets $V(\varepsilon,\delta)$,
$\varepsilon,\delta>0,$ is called the measure topology $t_\tau$ on $S(\cM,\tau)$ \cite{DPS, FK, Nelson}.
It is well known that the algebra $S(\cM,\tau)$ equipped with the measure topology is a complete metrizable topological algebra \cite{Nelson}.
We note that a sequence $\{x_n\}_{n=1}^\infty\subset S(\cM,\tau)$ converges to zero with respect to measure topology $t_\tau$ if and only if $\tau\big( e  ^{|x_n|}(\varepsilon,\infty)\big)\to 0$ as $n\to \infty$ for all $\varepsilon>0$ \cite{DPS}.

The space  $S_0(\cM,\tau)$ of $\tau$-compact operators is the space associated to the algebra of functions from $S(0,\infty)$ vanishing at infinity, that is,
$$S_0(\cM,\tau) = \{x\in S(\cM,\tau) :  \ \mu(\infty; x) =0\}.$$
The two-sided ideal $\cF(\tau)$ in $\cM$ consisting of all elements of $\tau$-finite range is defined by
$$\cF(\tau)=\{x\in \cM ~:~ \tau(r(x)) <\infty\} = \{x \in \cM ~:~ \tau(s(x)) <\infty\}.$$
Note that  $S_0(\cM,\tau)$ is the closure of $\cF(\tau)$ with respect to the measure topology \cite{DP2014}.


\section{The BCP of von Neumann algebras}
In this section, we fully characterize those von Neumann algebras having the  BCP/UBCP.

Before proceeding to the main result of this section, we need several auxiliary results.

A positive linear functional $\varphi$ on a Neumann algebra $\cM$  is said to be   pure if every positive linear functional $\psi$ on $\cM,$
majorized by $\varphi$ in the sense that $\psi(x^\ast x) \le  \varphi(x^\ast x)$ for all $x\in \cM,$ is of the form $\lambda \varphi,\,  0 \le \lambda \le  1$\cite[Chapter I, Definition 9.21]{Tak}. The following proposition   is certainly known to experts.
Due to the lack of suitable references, we provide a complete proof below. 

\begin{proposition}\label{atomlessproj}
Let $\cM$ be an atomless von
Neumann algebra and let $x\in \cM.$ Then there is a singular state $\varphi \in \cM$ such that
\begin{align*}
\varphi(x^\ast x)=\left\|x^\ast x\right\|_\infty=\|x\|_\infty^2.
\end{align*}
\end{proposition}
\begin{proof}
Let $\cN$ be a masa of $\cM$  containing $x^* x$ and let $\cE $ be the conditional expectation from $\cM$ onto $\cN$ \cite[Theorem 8.3 and Lemma 8.4]{Davidson}.
By \cite[II.6.3.3]{Blackadar} (see also \cite[Theorem 4.3.8]{KR}), there exists a pure state $\phi$ on $\cN$ such that
$$\phi(x^* x)=\norm{x}_\infty^2 .$$
Let $e$ be an arbitrary projection in $\cN$.
Since $\cM$ is atomless, it follows that
there exist mutually orthogonal non-trivial  projections $e_1$ and $e_2$ in $\cN$ such that $$e_1+e_2=e . $$
We claim that either $\phi(e_1)$ or $\phi(e_2)$ is zero.
Without loss of generality, we may assume that  $\phi(e_1)>0$.
Note that $\phi(\cdot e_1)$  is  majorized by $\phi$.
Since $\phi$ is pure, it follows from the definition  that $\phi(\cdot e_1)=\lambda \varphi(\cdot)$ for some $0\le \lambda\le 1 $. On the other hand, the assumption $\phi(e_1)>0$ implies that $\lambda =1$.
Therefore, $\phi(\cdot e_1)= \varphi(\cdot)$.
In other words, $\phi(e_2)=0$.
By \cite[Chapter III, Theorem 3.8]{Tak}, $\phi$ is singular.
Define  $$\varphi(\cdot):=\phi(\cE(\cdot  )) = \phi(\cdot).$$
By \cite[III.2.1.11]{Blackadar},
 $\phi(\cE(\cdot ))$   is singular on $\cM$.
Observe that $\varphi(x^\ast x)= \phi(x^* x)=\norm{x}_\infty^2$. The proof is complete.
\end{proof}
%
%
%

The following proposition is a simple consequence of the Cauchy--Schwarz inequality\cite[Chapter I, Proposition 9.5]{Tak}.
\begin{proposition}\label{projnorm}
Let $\cM$ be a von
Neumann algebra, $x\in \cM$ and $\varphi  $ be a state on $\cM$ such that
$\varphi(x^\ast x)=\left\|x^\ast x\right\|_\infty=\left\|x\right\|_\infty^2.$ For any   $p\in \cP(\cM)$ with $\varphi(p)=1,$
we have
$$
\left\|xp\right\|_\infty =\left\|x\right\|_\infty
$$
\end{proposition}

\begin{proof}  Let $p\in \cP(\cM)$ be such that $\varphi(p)=1.$
Let  $q:=\mathbf{1}-p.$
By the  Cauchy--Schwarz inequality\cite[Chapter I, Proposition 9.5]{Tak},  for any
$y\in \cM$,  we have
$$
|\varphi(y q)|^2\le \varphi(q)\varphi(y y^\ast)=(1-\varphi(p) )\varphi(y y^\ast) =0 .
$$
That is,
$
\varphi(y q)=0.
$
A similar argument yields that
\begin{align}\label{qqq}
\varphi(q y)=\varphi(y q)=0.
\end{align}
On the other hand, we have
\begin{align}\label{pxp}
\left\|x^\ast x \right\|_\infty & =\left\|p \right\|_\infty \left\|x^\ast x\right \|_\infty \left \|p\right\|_\infty
\ge \left\|p x^\ast x p \right\|_\infty \ge  \varphi(p x^\ast x p)\nonumber\\
& =
\varphi(x^\ast x-qx^\ast x-x^\ast x q +qx^\ast x q)\stackrel{\eqref{qqq}}=\varphi(x^\ast x)=\left\|x^\ast x\right\|_\infty.
\end{align}
Hence, we have
$$
\norm{x p}_\infty^2 =\norm{(xp)^\ast (xp)}_\infty= \norm{px^\ast x p}_\infty\stackrel{\eqref{pxp}}{=} \norm{x^\ast x}_\infty=\norm{x}_\infty^2.
$$
The proof is completed.
\end{proof}

The following lemma is the key ingredient in the proofs of Proposition \ref{prop:finite} and Theorem \ref{atomlessnotbcp} below.
\begin{lemma}\label{atomlesscenter}
Let $\cM$ be an atomless  von Neumann algebra and $\{x_n\}_{n\ge 1}$ be a sequence in $\cM.$
There exists a net $\{p_j\}_{j \in \cJ}\subset \cP(\cM)$ such that  $p_j\downarrow 0$ and $
\left\|x_np_j\right\|_\infty =\left\|x_n\right\|_\infty
$ for all $n\ge 1.$
\end{lemma}

\begin{proof}
By Proposition~\ref{atomlessproj}, for  each $n\ge 1$,   there is a singular state $\varphi_n$
on $M$  that satisfies that
$$
\varphi_n(x^\ast_n x_n)=\left\|x_n\right \|_\infty^2.
$$
By \cite[Theorem 10.1.15 (ii)]{KR-II}, the functional
$$
\varphi:=\sum\limits_{n=1}^\infty \frac{1}{2^n}\varphi_n
$$
is also a  singular state on $\cM$.

Set  $$\cP(\cM)_\varphi:=\left\{p\in \cP(\cM): \varphi(p)=0\right\}.$$
By \cite[Chapter III, Theorem 3.8]{Tak},    $\cP(\cM)_\varphi$ is non-trivial.
Let $\mathcal{Q}$  be the system of all sets (finite or infinite)
of all pairwise orthogonal projections from $\cP(\cM)_\varphi$,
ordered by set inclusion.
  By Zorn's  Lemma, there exists a maximal element $\{q_i\}_{i\in I}$
in $\mathcal{Q}.$
We claim that $\bigvee\limits_{i\in I}q_i=\mathbf{1}.$
Assume by contradiction that
$\bigvee\limits_{i\in I}q_i\neq \mathbf{1}.$
By \cite[Chapter III, Theorem 3.8]{Tak},
there exists a non-zero projection $q\le \mathbf{1}-\bigvee\limits_{i\in I}q_i$
such that $\phi(q)=0.$
We have
$$
\{q_i\}_{i\in I} \subsetneqq
\{q\}\cup \{q_i\}_{i\in I}\in \mathcal{Q},
$$
which contradicts with the maximality of $\{q_i\}_{i\in I}.$ Hence,
$\bigvee\limits_{i\in I}q_i=\mathbf{1}.$

Let $\mathcal{F}$ be the system of all finite subsets of $I$ (ordered by inclusion). 
For each $F\in \mathcal{F}$, we define
$$p_F: =\mathbf{1}-\sum\limits_{i\in F} q_i.$$ We have
$\phi(p_F)=1$ for each  $F\in \mathcal{F}$ and $p_F \downarrow 0.$
In particular,
 $\varphi_n(p_F) = 1$ for all  $n \ge  1$ and $F\in \cF$.
By Proposition~\ref{projnorm}, we obtain that 
$\left\|x_np_F \right\|_\infty =\left\|x_n\right\|_\infty
$ for every $n\ge 1$ and $F\in \cF$.
\end{proof}

\begin{proposition}\label{prop:finite}
Let $\cM$ be an atomless von Neumann algebra equipped with a faithful normal state $\omega$.
Then, for any sequence $(y_n)_{n=1}^\infty$ in $\cM$ and $\varepsilon>0$, there exists an element $g\in \cM$ such that 
$$\norm{g-y_n}_\infty >\norm{y_n}_\infty +1-\varepsilon.$$
\end{proposition}
\begin{proof}
By Proposition \ref{atomlesscenter}, there exists a net $\{p_i\}_{i\in\cI}$ in $\cP(\cM)$ such that $p_i\downarrow 0$ and $\norm{y_np_i}_\infty =\norm{y_n}_\infty$ for all $n\ge 1$.
Let  $p_1 $ be an arbitrary projection in $  \{p_i\}_{i\in\cI}$.
Since $p_i\downarrow 0$, it follows that there exists $p_2$ in $  \{p_i\}_{i\in\cI}$ with $p_2\le p_1$ such that
$$\norm{y_1 (p_1-p_2)}_\infty >\norm{y_1}_\infty -\varepsilon/2. $$
Continue this process, we obtain a decreasing sequence $\{p_n\}_{n\ge 1}\subset\{p_i\}_{i\in\cI} $ such that
$$\norm{y_n (p_n-p_{n+1})}_\infty >\norm{y_n}_\infty -\varepsilon/2. $$
By Proposition \ref{atomlesscenter} again, there exists a net $\{q_j\}_{j\in\cJ}$ in $\cP(\cM)$ such that $q_j\downarrow 0$ and $\norm{q_jy_n (p_n-p_{n+1})}_\infty =\norm{y_n (p_n-p_{n+1})}_\infty$ for all $n\ge 1$.
Arguing similarly, we   obtain a decreasing sequence $\{q_n\}_{n\ge 1}\subset\{q_j\}_{j\in\cJ} $ such that
$$\norm{(q_n-q_{n+1})y_n (p_n-p_{n+1})}_\infty >\norm{y_n}_\infty -\varepsilon. $$

Let $(q_n-q_{n+1})y_n (p_n-p_{n+1}) =u_n |(q_n-q_{n+1})y_n (p_n-p_{n+1}) |$ be the polar decomposition.
We define
$$g:= - \sum_{n\ge 1} (q_n-q_{n+1}) u_n (p_n-p_{n+1}). $$
We have
\begin{align*}
\norm{g-y_n}_\infty& \ge \norm{(q_n-q_{n+1})g (p_n-p_{n+1}) -(q_n-q_{n+1})y_n (p_n-p_{n+1})   }_\infty\\
& =\norm{   -(q_n-q_{n+1})  u_n(p_n-p_{n+1})-(q_n-q_{n+1})y_n (p_n-p_{n+1})  }_\infty  \\
&=\norm{u_n((p_n-p_{n+1})+ |(q_n-q_{n+1})y_n (p_n-p_{n+1}) |) }_\infty \\
&>\norm{y_n}_\infty +1- \varepsilon.
\end{align*}
This completes the proof.
\end{proof}

Now, we are ready to prove the main result of this section, which shows that all atomless von Neumann algebras fail the BCP.
\begin{theorem}\label{atomlessnotbcp}
Let $\cM$ be a non-trivial   atomless von Neumann algebra.
Then, for any sequence $(y_n)_{n=1}^\infty$ in $\cM$ and $\varepsilon>0$, there exists an element $g\in \cM$ such that 
$$\norm{g-y_n}_\infty >\norm{y_n}_\infty +1-\varepsilon.$$
In particular, any  non-trivial   atomless   von
Neumann algebra  fails the  BCP.
\end{theorem}

\begin{proof}Let $\omega$ be a semifinite faithful normal weight on $\cM$.
For each $n\ge 1$,  we may choose  a non-trivial projection  $p_n\in \cM$   such that $p_n\le e^{|y_n |}(\norm{y_n}_\infty-\varepsilon,\norm{y_n}_\infty]$ with $\omega(p_n)<\infty$.
On the other hand, we may choose  a non-trivial  projection  $q_n\in \cM$   such that $q_n\le e^{|y_n^*|}(\norm{y_n}_\infty-\varepsilon,\norm{y_n}_\infty]$ with $\omega(q_n)<\infty$.
In particular, $p_n$ and $q_n$ are
$\sigma$-finite and $$Q:=(\vee_{n\ge 1}p_n)\vee (\vee_{n\ge 1}q_n)$$ is also $\sigma$-finite (see e.g. \cite[Proposition 5.5.9 and 5.5.19]{KR}).
Note that $$\norm{y_n- g }_\infty \ge \norm{Q(y_n -g)Q }_\infty $$
for any $g\in \cM$ (in particular, for any $g\in \cM_Q=Q\cM Q$). It is sufficient to prove that there exists $g\in \cM_Q $ such that \begin{align}\label{finite}
\norm{Qy_n Q -g  }_\infty > \norm{Q y_n Q}_\infty +1-\varepsilon\ge  \norm{y_n }_\infty +1-2\varepsilon .
\end{align}
Therefore, without loss of generality, we may assume that $\cM$ is a $\sigma$-finite Neumann algebra  equipped with a faithful normal state\cite[Ex. 7.6.46]{KR-II}.
By Proposition~\ref{prop:finite}, we obtain the existence of $g\in  \cM_Q$ in
Eq. \eqref{finite}.
This competes the proof.
\end{proof}

Let $\cM$ be a von Neumann algebra equipped with a   semifinite faithful normal trace $\tau$.
The so-call $\tau$-compact ideal $C_0(\cM,\tau)$  is the uniform norm closure of
$\cF(\tau)$ in $\cM$~\cite[Definition 2.6.8]{LSZ}. Equivalently, $$C_0(\cM,\tau)=S_0(\cM,\tau)\cap \cM$$
see e.g. \cite[Lemma 2.6.9]{LSZ}, \cite{DP2014}.
The   ideal  $\cJ(\cM)$ of compact operators (i.e., the uniform norm closed two sided ideal   generated by the finite projections of a semifinite von Neumann algebra)  was introduced by Kaftal~\cite{Kaftal}.
Both $C_0(\cM,\tau)$ and $\cJ(\cM)$ are natural generalization of the ideal $K(\cH)$ of all compact operators on a separable Hilbert space $\cH$.
Whenever $\cM$ is a factor, we have $C_0(\cM,\tau)=\cJ(\cM)$, see e.g. \cite[Remark 2.7 and Theorem 2.8]{BHLS} and \cite[Theorem 1.3]{Kaftal}.
In particular, when $\cM=B(\cH)$, we have
$$C_0(\cM,\tau)=\cJ(\cM)=K(\cH). $$

Proposition \ref{prop:finite} shows that any von Neumann algebra equipped with finite faithful normal trace fails the BCP.
Arguing similar to   the proof of Theorem \ref{atomlessnotbcp}, we obtain the following corollary. 

 \begin{cor}\label{th:atomless}
Let $\cM$ be an atomless von Neumann algebra equipped with a semifinite faithful normal trace $\tau$.
For  any sequence $(y_n)_{n=1}^\infty$ in $\cM$ and $\varepsilon>0$, there exists an element $g\in  \cF(\tau)$ such that 
$$\norm{g-y_n}_\infty >\norm{y_n}_\infty +1-\varepsilon.$$
Consequencely, none of the spaces  $$C_0(\cM,\tau) ,  ~\cJ(\cM)  \mbox{ and } \cM$$    has the BCP.
\end{cor}

\begin{remark}
Ozawa  and Vaes  independently  showed\cite{OV} that  for a sequence $\{x_n\}_{n\ge 1}$ of elements
in an atomless von Neumann algebra $\cM$ equipped with a semifinite faithful normal weight $\omega$, for any $\varepsilon>0$, there exists a projection $p$ such that $\omega(p)<\varepsilon$ and
\begin{align}\label{Vaes}\norm{x_n p}_\infty =\norm{x_n}_\infty
\end{align}
for all $n\ge 1$.
In particular, there exists a non-trivial projection $p$ such that
$$\norm{x_n-({\bf 1}-p)}_\infty \ge \norm{(x_n-({\bf 1}-p) ) p }_\infty  =\norm{x_np}_\infty =\norm{x_n}_\infty,$$
which shows that $\cM$ fails the BCP.
 Our proof  for Theorem \ref{atomlessnotbcp} is motivated by Ozawa's argument.
Corollary \ref{th:atomless} is a consequence of Theorem \ref{atomlessnotbcp}  and \eqref{Vaes}.
\end{remark}

The following theorem    characterizes those von Neumann algebras having the BCP.

\begin{theorem}\label{main1}
A   von Neumann algebra  $\cM$ has the BCP/UBCP if and only if
\begin{center}
  $\cM$ is an atomic  and $\cM_*$ is separable.
\end{center}
\end{theorem}
 \begin{proof} By \cite[Corollary 3.1]{LLLZ}, $B(\cH)$ has the UBCP whenever $\cH$ is a separable Hilbert space  (equivalently,  $\cM$ is an atomic  and $\cM_*$ is separable), which together with
  \cite[Theorem 2.8]{LZ2020} implies that  any atomic type $I$ von Neumann algebra on a separable Hilbert space has the BCP.
 The UBCP  of any atomic type $I$ von Neumann algebra on a separable Hilbert space  follows from Theorem \ref{main2} below.

By Proposition \ref{atomlessnotbcp}, it remains to show that any non-$\sigma$-finite atomic  von Neumann algebra fails the BCP.
The case when $\cM$ is commutative follows from \cite[Remark 2.3]{LZ2020}. The general case when $\cM$ is a  non-$\sigma$-finite von Neumann algebra follows from a similar argument, see e.g. Proposition~\ref{prop:noncou} below for a complete proof.
\end{proof}

The situation for $C^*$-algebras is quite different from that of von Neumann algebras.

Let $\cA$ be a unital commutative $C^\ast$-algebra and let $\mathfrak{P}(\cA)$ be the  space of all pure states of $\cA$.
The pure state space 
  $\mathfrak{P}(\cA)$ is a compact Hausdorff space with respect to 
  the 
weak$^\ast$-topology,    and,  by well-known Gelfand--Naimark Theorem \cite[Theorem 4.4.3]{KR} (see also \cite[Theorem II.2.2.4]{Blackadar}), 
the $C^\ast$-algebra $\cA$ is *-isomorphic to the $C^\ast$-algebra $C(\mathfrak{P}(\cA))$ of all complex-valued continuous functions on $\mathfrak{P}(\cA).$

Let $K$ be a topological space. A collection of nonempty open subsets
$\mathcal{U} = \left\{U_\alpha\right\}_{\alpha\in  A}$ of $K$ is called a $\pi$-basis (weak basis) for $K$
if every nonempty open subset
of $K$ contains some $U_\alpha$ in $\mathcal{U}$ \cite[Page 15]{T}.

 In the setting of commutative $C^*$-algebra,  \cite[Theorem 2.1]{LLLZ} can be reformulated as follows. Note that \cite[Theorem 2.1]{LLLZ} only deals with real continuous functions. The complex case is a direct consequence of the real case.

\begin{corollary}\label{com}
Let $\cA$ be a  unital commutative $C^\ast$-algebra. Then the followings are equivalent:
\begin{enumerate}
\item $\cA$  has the BCP;
\item the pure state space $\mathfrak{P}(\cA)$  has a countable $\pi$-basis;
\item $\cA$ has the UBCP.
\end{enumerate}
\end{corollary}
In particular, there exist   commutative $C^*$-algebras containing no minimal projections but  having the BCP. 
For example, the uniform norm closure of the linear span of Haar functions\cite[p.150]{LT2}.

We conclude this section  by the following question.

\begin{question}
 It will be interesting to find 
  a non-commutative version of Corollary~\ref{com}, that is, how to 
characterize  those  $C^\ast$-algebras $\cA$ having  BCP/UBCP. 
\end{question}

\section{The BCP of noncommutative symmetric spaces}

Let $E(0,\infty)$  be a Banach space of real-valued Lebesgue measurable
functions on  $(0,\infty)$ (with identification
$m$-a.e.), equipped with a  norm $\left\|\cdot\right\|_E$.
The space $E(0,\infty)$ is said to be {\it
absolutely solid} if $x\in E(0,\infty)$ and $|y|\leq |x|$, $y\in S(0,\infty)$
implies that $y\in E(0,\infty)$ and $\|y\|_E\leq\|x\|_E.$
An absolutely solid space $E(0,\infty)\subseteq S(0,\infty)$ is said to be {\it
symmetric} if for every $x\in E(0,\infty)$ and every $y\in S(0,\infty)$,
 the assumption
$\mu(y)=\mu(x)$ implies that $y\in E(0,\infty)$ and $\left\|y\right\|_E=\left\|x\right\|_E$
\cite{KPS,LT2,Bennett_S}.

\begin{definition}\label{opspace}
Let $\cM $ be a semifinite von Neumann  algebra  equipped
with a faithful normal semi-finite trace $\tau$.
Let $\mathcal{E}$ be a linear subset in $S({\mathcal{M}, \tau})$
equipped with a complete norm $\left\|\cdot \right \|_{\mathcal{E}}$.
We say that
$\mathcal{E}$ is a \textit{symmetric    space}  if
for $x \in
\mathcal{E}$, $y\in S({\mathcal{M}, \tau})$ and  $\mu(y)\leq \mu(x)$ imply that $y\in \mathcal{E}$ and
$\left\|y\right\|_\mathcal{E}\leq \left\|x\right\|_\mathcal{E}$.
\end{definition}
Let $E(\cM,\tau)$ be a symmetric   space.
It is well-known that any symmetrically normed space $E(\cM,\tau)$ is a normed $\cM$-bimodule (see e.g.  \cite{DP2014} and \cite{DPS}). 
That is, for any symmetric operator space $E(\cM,\tau)$, we have
$$\left\|axb\right\|_E \le \left\|a\right\|_\infty \left\|b\right\|_\infty \left\|x\right\|_E, ~a, b \in \cM,~ x\in E(\cM,\tau) .$$
It is known that whenever $E(\cM,\tau)$ has   order continuous norm $\norm{\cdot}_E$, i.e.,  $\norm{x_\alpha}_E\downarrow 0$ whenever $0\le x_\alpha \downarrow 0\subset E(\cM,\tau)$, we  have $E(\cM,\tau)\subset S_0(\cM,\tau)$~\cite{DPS,HSZ,DP2014}.
Moreover, if $\cM$ can be represented on a separable Hilbert space, then $E(\cM,\tau)$ is separable whenever $\norm{\cdot}_{E}$ is order continuous on $E(\cM,\tau)$  \cite{Me}. 
In particular, a symmetric function space $E(0,\infty)$ is separable if and only if it has order continuous norm\cite[page 7]{LT2}.

There exists a strong connection between symmetric function spaces and
operator spaces exposed in \cite{Kalton_S} (see also \cite{LSZ,DP2014}).
The operator space $E(\cM,\tau)$ defined by
\begin{equation*}
E(\mathcal{M},\tau):=\{x \in S(\mathcal{M},\tau):\ \mu(x )\in E(0,\infty)\},
\ \left\|x \right\|_{E(\mathcal{M},\tau)}:=\left\|\mu(x )\right\|_E
\end{equation*}
 is a complete symmetric  space  whenever $(E(0,\infty),\left\|\cdot\right\|_E)$ is    a complete  symmetric  function space on $(0,\infty)$  \cite{Kalton_S}.
In particular, for any symmetric function space $E(0,\infty)$, we have \cite[Lemma 18]{DP2014} $$F(\tau)\subset E(\cM,\tau). $$
In the special case when $E(0,\infty)=L_p(0,\infty)$, $1\le  p\le \infty$, $E(\cM,\tau)$ is the noncommutative $L_p$-spaces affiliated with $\cM$ and we denote the norm by $\norm{\cdot}_p$.

The \emph{carrier projection} $c_E=c_{E(\cM,\tau)}\in \cM$ of   $E(\cM,\tau)$ is defined by setting
$$c_E := \bigvee \left\{p:p\in \cP(E)\right\}.$$
 It is clear that $c_E$ is in the center $Z(\cM)$ of $\cM$ \cite{DPS,DP2014}.
It is often assumed that the carrier projection $c_E$ is equal to ${\bf 1}$.
Indeed, for any  symmetric function  space  $E(0,\infty)$, the carrier projection of the corresponding operator space $E(\cM,\tau)$ is always ${\bf 1}$ (see e.g.
\cite{DP2014,Kalton_S,DPS}).
On the other hand, if $\cM$ is atomless or is atomic and all atoms have equal trace, then any non-zero symmetric space $E(\cM,\tau)$ is necessarily $\bf 1$~\cite{DP2014,DPS}.

If $x,y\in S(\cM,\tau)$, then $x$ is said to be submajorized by $y$, denoted by $x\prec\prec y$, if \begin{align*} \int_{0}^{t} \mu(s;x) ds \le \int_{0}^{t} \mu(s;y) ds \end{align*} for all $t\ge 0$.
 A symmetric space $E(\cM,\tau)\subset S(\cM,\tau)$ is called \emph{strongly symmetric} if its norm $\left\|\cdot\right\|_E$ has the additional property that   $\left\|x\right\|_E \le \left\|y\right\|_E$ whenever $x,y \in E(\cM,\tau)$ satisfy $x\prec\prec y$.
In addition, if $x\in S(\cM,\tau)$, $y \in E(\cM,\tau)$ and $x\prec\prec y$ imply that $x\in E(\cM,\tau)$ and $\|x\|_E \le \|y\|_E$, then
 $E(\cM,\tau)$  is said to be a \emph{fully symmetric space} (of $\tau$-measurable operators).
Recall that every separable symmetric sequence/function  space $E$ is necessarily fully symmetric  (see e.g. \cite{LT2} \cite[Chapter II,Theorem 4.10]{KPS} or \cite[Chapter IV, Theorem 5.7]{DPS}).
Also,   most familiar symmetrically normed spaces, including the usual $L_p$-spaces and those associated with the names of Orlicz, Lorentz and Marcinkiewicz   are fully symmetric.

 It is known that $\ell_p(\Gamma)$ does not have the BCP for any $1\le p \le \infty$ and any uncountable index set, see e.g. \cite[Remrak 2.3]{LZ2020}.
 The following proposition extends this to the setting of noncommutative symmetric spaces.
 \begin{proposition}\label{prop:noncou}
 Let $\cM$ be a von Neumann algebra equipped with a semifinite faithful normal trace
 $\tau$.
 If  $\cM$ is not $\sigma$-finite, then any non-trivial symmetric space $E(\cM,\tau)$ with $c_{E(\cM,\tau)}={\bf 1 }$ fails the BCP.
 \end{proposition}
 \begin{proof}
 Since $\cM$ is not $\sigma$-finite, it follows that trace $\tau$ is infinite.

Let  $\{x_n\}_{n\ge 1}$ be an arbitrary sequence  in $E(\cM,\tau)$.
For any $n\ge 1 $, let
\begin{align*} r_n: =
\begin{cases}
 e^{|x_n|}(\mu(\infty; x_n),\infty), & \mbox{if } \mu(t; x_n)>   \mu(\infty; x_n) \mbox{ for all }t>0 , \\
 e^{|x_n|}(\mu(\infty; x_n),\infty) \vee w_n , & \mbox{otherwise},
\end{cases}
\end{align*}
where $w_n:= \vee _k z_k $, $z_n<e^{|x_n|  } (\mu(\infty; x_n)-\frac1k, \mu(\infty; x_n)]$ is a $\tau$-finite projection such that    $\tau(z_n)\to    \infty $ as $n\to \infty$.
Note that $r_n$  is a $\sigma$-finite projection in $\cM$ (see e.g. \cite[Corollary 5.5.7]{DPS}).

Letting $p: =\vee_n r_n$, we have
$$\mu(  x_n p)=\mu(x_n)$$
 for all $n$'s.
 Let $q={\bf 1}-p$. Note that $p$ is $\sigma$-finite and ${\bf 1}$ is not $\sigma$-finite. We obtain that  $q\ne 0 $.
 There exists a non-trivial $\tau$-finite projection $q'\le q$.
Since $q'\in E(\cM, \tau)$\cite[Lemma 18]{DP2014} and 
 $$\norm{x_n-q'}_E \ge \norm{(x_n-q')p}_E =  \norm{ x_np}_E =\norm{x_n}_E,$$
it follows that  $E(\cM,\tau)$ fails the BCP.
 \end{proof}
 \begin{remark}
 Note that the condition $c_{E(\cM,\tau)}={\bf 1}$ in Proposition \ref{prop:noncou}
can not be removed. For example,
$L_1(0,1/2)\oplus 0$ is a separable symmetric space affiliated with the algebra $L_\infty(0,1/2) \oplus \ell_\infty (\Gamma)$, where $\Gamma$ is an uncountable index set with the counting measure.

 \end{remark}

 Note that there are $\sigma$-finite von Neumann algebras which can not be represented on a separable Hilbert space, e.g. $\oplus_{\Gamma} L_\infty (0,1)$, where $\Gamma $ is an uncountable index set~\cite[Example 5.6.13]{DPS}.
Below, we show that
if $\cM $ is a $\sigma$-finite semifinite von Neumann algebra with a non-separable predual, then any
  fully symmetric spaces affiliated with  $\cM $
    fails the BCP.

\begin{proposition}\label{prop:nonseparable}
Let $\cM$ be a $\sigma$-finite  von Neumann algebra with a semifinite faithful normal   trace $\tau.$
Let $E(\cM,\tau)$ be a  non-trivial fully symmetric space affiliated with $\cM$.
If $\cM_*$ is not separable, then
$E(\cM,\tau)$ fails the BCP.
\end{proposition}

\begin{proof}
Since $E(\cM,\tau)$  is a fully symmetrically normed space and $E(\cM,\tau)\ne \{0\}$, it follows that  $c_{E(\cM,\tau)}={\bf 1}$,  see e.g. \cite[Remark 4.5.4(b)]{DPS}.

Let $\{x_n\}_{n\ge 1}$ be an arbitrary sequence in $E(\cM,\tau)$.
 Let $$p:=\bigvee_n \left(l(x_n)\vee r(x_n)\right).$$
If $p\ne {\bf 1}$, then we have
\begin{align}\label{pne1}
\norm{x_n - \frac{q}{\norm{q}_E}}_{E }\ge \norm{x_n}_E
\end{align}
for all $n$ and any $\tau$-finite projection $q\in \cM$ with  $q\le {\bf 1}-p$.

Assume that $p={\bf 1} $.
Since $\cM$ is $\sigma$-finite, it follows that there exists a sequence $\{p_k\}_{k\ge 1}$ of $\tau$-finite projections in $\cM$ such that $$p_k\uparrow {\bf 1} .$$
Let $\cN$ be the von Neumann subalgebra   of $\cM$ generated by
$$
\{ x_{n,k}\}_{n,k \ge 1}\cup \{p_n\}_{n\ge 1}$$
 in $\cM,$ where $x_{n,k}:=x_n e^{|x_n|}[k-1, k)$,  $k\ge 1.$
Moreover, since $\cM$ is $\sigma$-finite, it follows that   $\cN$   can be represented on a   separable Hilbert space\footnote{It is well-known that a countably generated $\sigma$-finite von Neumann algebra can be represented on a separable Hilbert space. Indeed,
since $\cN$ is $\sigma$-finite, by \cite[Proposition II 3.19]{Tak}, $\cN$ admits a faithful normal state $\rho.$
Let $H$ be a Hilbert space completion of $\cN_\rho=\{x\in \cN: \rho(x^\ast x)<\infty\}$. Take a  $\ast$-subalgebra $\cA$ in $\cN$ generated by the countable set $\{x_{n,k}\}_{k\ge1}\cup\{p_n\}_{n\ge 1}.$ By Kaplansky Density Theorem \cite[Theorem II 4.8]{Tak}, the unit ball
of $\cA$ is dense in the unit ball of $\cN$ in the strong$^\ast$ topology.
By \cite[Proposition III 5.3]{Tak}, the  strong$^\ast$ topology on the unit ball of $\cN$ coincides with the topology generated by $\left\|\cdot\right\|_{\rho,2}$-norm (the Hilbertian norm generated by 
$\rho$). Since $\cA$ is countable generated, it follows that the unit ball of $\cN$ is $\left\|\cdot\right\|_{\rho, 2}$-separable. Therefore,   $\cN$ is also $\left\|\cdot\right\|_{\rho, 2}$-separable. So,   $\cN$ can be represented on  a separable Hilbert space  $H$ via left multiplication.}.
Note that $\tau|_\cN$ is also semifinite.
Indeed,
 for any $0<a\in\cN_+$, there exists $p_k$ such that  $\tau(a^{1/2} p_k  a^{1/2} )<\infty $
and  $0< a^{1/2} p_i  a^{1/2} \le a$, in other words, $\tau$ is semifinite on $\cN$.

Since $\cN$ can be represented on a separable Hilbert space, it follows from  \cite[Proposition 5.6.17 and Theorem 5.6.22]{DPS} (see also \cite[Proposition 1.2]{S00} and \cite{Me}) that
 any   symmetric space having order continuous norm affiliated with $\cN$ is separable.
Since $\cM_*$ is not separable, it follows that  $\cN \subsetneqq \cM$.
Let $\cE$ be the conditional expectation from $\cM$ onto $\cN$~\cite[Lemma 2.1]{DDPS} (see also \cite[Chapter V, Proposition 2.36]{Tak}).

Let us take an arbitrary  non-zero element $y\in \cM$ such that $y\notin \cN.$ Since $y\ne \cE(y)\in \cN,$ we  define an element $x\in \cM$ by setting
$$
x:=\frac{y-\cE(y)}{\|y-\cE(y)\|_1} \ne 0.
$$
By the  construction of $\cN$,
we have that $\{x_n\}_{n\ge 1}\subset L_1(\cN, \tau|_\cN).$ Taking into account that $\cE(x_n)=x_n$, $n\ge 1$,  $\cE(x)=0$ and $E(\cM,\tau)$ is fully symmetric,  we obtain
\begin{align*}
\left\|x-x_n\right\|_1 &   \stackrel{\mbox{\tiny \cite[Proposition~2.1]{DDPS}}}{\ge} \left\|\cE(x-x_n)\right\|_1=\left\|x_n\right\|_1
\end{align*}
for all $n\ge 1.$
This together with \eqref{pne1} shows that  $E(\cM,\tau)$ fails the  BCP. The proof is complete.
\end{proof}

%
%
%
%
%

By Propositions \ref{prop:noncou} and \ref{prop:nonseparable}, it suffices to consider symmetric spaces affiliated with a semifinite von Neumann algebra whose predual is separable.

Recall that a symmetric function space $E=E(0,1)$ (or $E(0,\infty)$) is said to have the Fatou norm if for any upwards directed net $\{x_\alpha\}$ in $E^+$ and $x\in E^+$, it follows from $x_\alpha\uparrow x$ that $$\norm{x_\alpha}_E\uparrow_\alpha\norm{x}_E .$$

Let $(\cM,\tau)$  be a semifinite von Neumann algebra. The trace $\tau$ is said to be separable if the set
$$P(\cM)\cap \cF(\tau)=\{p\in P(\cM):\tau(p)<\infty\}$$
is separable with respect to the measure topology.
If $\tau$ is separable, then $\cM$  is $*$-isomorphic to a von
Neumann subalgebra of $B(K)$ for some separable Hilbert space $K$, see
\cite[Proposition 1.1]{S00} and \cite[Theorem 5.6.22]{DPS}.
On the other hand, if $\cM$ acts on a separable Hilbert space, then
the trace $\tau$ is separable \cite[Proposition 5.6.17]{DPS}.

 The subset $E(\cM,\tau)^{oc}$ of $E(\cM,\tau)$ is defined by setting
$$E(\cM,\tau)^{oc} =\left\{
x\in E(\cM,\tau): |x| \ge x_n \downarrow _n 0 \Rightarrow \norm{x_n}_E\downarrow_n 0
 \right\}.$$
 The elements of $E(\cM,\tau)^{oc}$ are called the elements of order continuous norm in $E(\cM,\tau)$, see \cite[Section 8.4]{DP2014} or \cite[Chapter 5.4]{DPS}.
 We denote the closure of $\cF(\tau)$ in $E(\cM,\tau)$ by $E(\cM,\tau)^b$, that is,
 $$E(\cM,\tau)^b =\overline{\cF(\tau)}^b. $$


Below, we provide a characterization for a symmetric space to have the BCP.
\begin{theorem}\label{main}
Let $\cM$ be 
a von Neumann algebra  equipped with a separable semifinite faithful normal trace $\tau$.
Let $E(\cM,\tau)$ be a strongly symmetric space affiliated with 
$\cM$ such that  $\norm{\cdot}_E$ is a  Fatou norm.
If $$E(\cM,\tau)^{oc}=E(\cM,\tau)^b ,$$  then
$E(\cM,\tau)$ has the UBCP.
\end{theorem}

\begin{proof}
Without loss of generality, we may  assume that  $c_{E(\cM,\tau)}={\bf 1}$.

Since $\tau$ is separable, it follows that
$E(\cM,\tau)^{oc}$ is separable  \cite[Theorem 5.6.22]{DPS}.
Therefore, there exists a countable subset $\cA$ of $E(\cM,\tau)^{oc} $ such that for any $y\in E(\cM,\tau)^{oc}$ and any $\varepsilon>0$, there exists $z\in \cA$ such that $\norm{z-y}_E<\varepsilon$.

 Let $x$ be an arbitrary element in the  unit sphere of  $E(\cM,\tau)$ with polar decomposition $x=u|x|$.
Let $\varepsilon\in (0,\frac14)$.
Since $\norm{\cdot}_E$ is a Fatou norm, it follows that there exists a projection $p\in \cF(\tau) $   such that
\begin{align}\label{esti:x}
1=\norm{x}_E\ge \norm{u |x|^{1/2}p|x|^{1/2}}_E\ge \norm{x}_E -\varepsilon \ge \frac{3}{4} .
\end{align}
Note that $u |x|^{1/2}p|x|^{1/2} \in \cF(\tau)\subset E(\cM,\tau)^b =E(\cM,\tau)^{oc} \subset E(\cM,\tau) $ and $-{\bf 1}\le 2p-{\bf 1} \le {\bf 1}$.
Therefore,  we have \cite[Proposition 1(iii)]{DP2014}
$$-  |x| = - |x|^{1/2}  {\bf 1} |x|^{1/2 } \le |x|^{1/2} (2p-{\bf 1})|x|^{1/2} \le |x|^{1/2}  {\bf 1} |x|^{1/2} =  |x|. $$
It follows from \cite{FC88} (see also \cite[Lemma 4.2]{HSZ})
that
$$ |x|^{1/2} (2p-{\bf 1})|x|^{1/2}  \prec \prec  |x|. $$
We have $|x|^{1/2} (2p-{\bf 1})|x|^{1/2} =|x|^{1/2}  2p |x|^{1/2}-|x| \stackrel{\tiny \mbox{\cite[Prop.3(iii)]{DP2014}}}{\in} E(\cM,\tau)$ and
\begin{align}\label{2u}
\norm{2u |x|^{1/2}p|x|^{1/2}  - x}_E =\norm{ u |x|^{1/2}( 2p- {\bf 1} ) |x|^{1/2}}_E\le \norm{ x}_E =1   . 
\end{align}
By the separability of $ E(\cM,\tau)^{oc}$,
there exists an element $g\in \cA\subset E(\cM,\tau)^{oc}$ such that
$$\norm{g- 2u |x|^{1/2}p|x|^{1/2} }_E <\frac18   .$$
This together with \eqref{esti:x} implies that 
$$\frac{11}{8}\le   \norm{ 2u |x|^{1/2}p|x|^{1/2} }_E -\frac18 \le \norm{g }_E \le   \norm{ 2u |x|^{1/2}p|x|^{1/2} }_E +\frac18 \le\frac{17}{8}.$$
On the other hand, by the triangular inequality, we have 
\begin{align*}
\norm{g-x}_E &\le \norm{g- 2u |x|^{1/2}p|x|^{1/2} }_E+\norm{2u |x|^{1/2}p|x|^{1/2}  - x}_E   \\
&\stackrel{\eqref{2u}}{<}\frac18  + 1\\
&\le   \frac98 <\frac{11}{8}\le \norm{g}_E,
\end{align*}
which completes the proof.
\end{proof}

Recall that any symmetric function space having the Fatou norm is a strongly symmetric space, see e.g. \cite[Corollary 4.5.7]{DPS} and \cite{KPS}. The following result is a immediate consequence of \cite[Proposition 56]{DP2014} and Theorem \ref{main} above.
\begin{cor}
Let $E(0,\infty)$ be a symmetric function space whose norm is a Fatou norm, and let $\cM$ be a von Neumann algebra equipped with a
separable semifinite faithful normal trace $\tau$.
If $$E(0,\infty)^{oc}\ne \{0\} ,$$  then
$E(\cM,\tau)$ has the UBCP.
\end{cor}

\begin{cor}Assume  $\cM $ is   a von Neumann algebrea   equipped with a
semifinite faithful normal trace $\tau$. 
  Let $E(\cM,\tau)$ be a strongly symmetric space affiliated with $\cM$
 having order continuous norm. 
  Then, $E(\cM,\tau)$ has the UBCP/BCP if and only if it is separable.
  In particular, $L_p(\cM,\tau)$ has the UBCP/BCP if and only if it is separable if and only if $\cM$ can be represented on a separable Hilbert space.
\end{cor}
\begin{proof}
Without loss of generality, we may  assume that  $c_{E(\cM,\tau)}={\bf 1}$.

Note that every strong symmetric space having order continuous norm  is fully symmetric \cite[Corollary 5.3.6]{DPS}.

If $E(\cM,\tau)$ is not separable, then $\cM_*$ is not separable, see e.g. \cite{Me} or \cite[Theorem 5.6.22]{DPS}. 
By Proposition \ref{prop:nonseparable}, $E(\cM,\tau)$ fails the BCP. 
On the other hand, all separable Banach spaces have the UBCP. This completes the proof.
\end{proof}

For general symmetric spaces $E(\cM,\tau)$, we only have  $E(\cM,\tau)^{oc}\subset E(\cM,\tau)^b$~\cite[Proposition 46]{DP2014}.
The following lemma  characterizes those symmetric spaces  $E(\cM,\tau)$ such that   $E(\cM,\tau)^{oc}=E(\cM,\tau)^b $, see \cite[Proposition 5.4.8]{DPS}  or \cite[Proposition 48]{DP2014}.
\begin{lemma}\label{lemmaoc}
Let $E(\cM,\tau)$ be a strongly symmetric space affiliated with a semifinite faithful normal trace $\tau$.
Assume that $c_{E(\cM,\tau)}={\bf 1}$.
Consider the following two conditions:\begin{enumerate}
                                              \item  $\norm{p_n}_E\to 0$ whenever $\tau(p_n)\to 0$ as $n\to \infty$;

                                              \item $E(\cM,\tau)^{oc}=E(\cM,\tau)^b.$
                                            \end{enumerate}
The implication $(1)\Rightarrow (2)$ is always valid. If the von Neumann algebra is atomless, then we have   $(1)\Leftrightarrow (2)$.
\end{lemma}

\begin{cor}\label{atomlessnoinfty}
Let $\cM$ be an    atomless von Neumann algebra equipped with a separable semifinite faithful normal trace $\tau$.
Let $E(0,\infty)$ be a symmetric function space with a Fatou norm.
If the restriction of $E(0,\infty)$ on $(0,1)$ does not coincides with $L_\infty(0,1)$, then   $E(\cM,\tau)$ has   the UBCP.

\end{cor}
\begin{proof}
Note that a symmetric function space having a  Fatou norm is strongly symmetric, see e.g. \cite[Corollary 4.5.7]{DPS} and \cite{KPS}.
In particular, $E(\cM,\tau)$ is a strongly symmetric space having a Fatou norm\cite[Theorems 51 and 54]{DP2014} and $c_{E(\cM,\tau)}={\bf 1}$\cite[p.245]{DP2014}. 

Since  the restriction of $E(0,\infty)$ on $(0,1)$ does not coincides with $L_\infty(0,1)$, it follows that
\begin{align*}
\norm{\chi_{(0,\delta)}}_E\to 0
\end{align*}
as $\delta\to 0$, see e.g. \cite[Chapter 2.a]{LT2}.
By Lemma \ref{lemmaoc}, we have $$E(0,\infty)^{oc}=E(0,\infty)^{b}. $$
In particular, $E(0,\infty)^{oc}\ne \{0\}$. 
By \cite[Proposition 56]{DP2014}, we obtain that 
$$E(\cM,\tau)^{oc}=E(\cM,\tau)^{b}$$
Now, the assertion is an immediate consequence of Theorem \ref{main}.
\end{proof}

  \begin{corollary}
  If $X$ is a symmetric function space having the Fatou norm satisfying one of the followings:
  \begin{enumerate}
   \item $X=E(0,1)$  which does not coincides with $L_\infty(0,1)$;
    \item $X=E(0,\infty)$   which does not coincides with $L_\infty(0,1)$ on $(0,1)$,
     \end{enumerate}
  then $X$ has the UBCP.
\end{corollary}
Recall that for atomic von Neumann algebra with all minimal projections having the same   trace, the condition $E(\cM,\tau)^{oc}=E(\cM,\tau)^b$ is always true, see e.g. \cite[Proposition 49]{DP2014}.
The situation changes when the traces of   minimal projections having are not necessarily the same.
For example, if $\ell_\infty(\mathbb{N})$ is equipped with the measure $\nu$ such that $\nu(e_n)=\frac{1}{2^n}$, then $${\bf 1}\in E(\cM,\tau)^b=\cF(\tau)  \mbox{ but } {\bf 1}\notin E(\cM,\tau)^{oc} . $$ 
However, the following result shows that for an atomic  von Neumann algebra $\cM$, the Fatou norm is sufficient  for a symmetric space $E(\cM,\tau)$ to have the UBCP.

\begin{theorem}
\label{main2}
Let $\cM$ be an atomic  von Neumann algebra equipped with a separable semifinite faithful normal trace $\tau$. Assume that     $E(\cM,\tau)$ is  a strongly symmetric  space with a Fatou norm,
then $E(\cM,\tau)$ has   the UBCP.
\end{theorem}
\begin{proof}

Without loss of generality, we may  assume that  $c_{E(\cM,\tau)}={\bf 1}$.

Note that an atomic von Neumann has the following representation:
$$\oplus_{k=1}^\infty \left(\mathcal{M}_k\otimes \mathcal{A}_k\right),$$
where $\mathcal{A}_k$'s are commutative atomic algebras.
The trace on $\cM$ is given by
$$\oplus_{k=1}^\infty \left(  {\rm Tr}_k \otimes \tau_k \right),$$
where $\tau_k$ is the trace on   $\cA_k$ and $ {\rm Tr}_k$ is the standard trace on the matrix algebra $\mathbb{M}_k$.
Since $\tau$ is separable, it follows that $L_1(\cM,\tau)$ is separable and the underlying Hilbert space can be chosen to be separable \cite[Proposition 1.1]{S00}.
Define a semifinite faithful normal trace  $\tau_k'
$ on $\cA_k$ by setting
$$\tau_k'(p)=\begin{cases}
                 \tau_k(p), & \mbox{if } \tau_k(p)\ge 1  \\
                 1, & \mbox{otherwise}
               \end{cases}$$
               for any minimal projection $p$ in $\cA_k$.
               In particular, $\tau'={\rm Tr_k}\otimes \tau_k'$ is a semifinite faithful normal trace on $\cM$.

Since $L_1(\cM,\tau')$ is separable (which is isometric to $L_1(\cM,\tau)$\cite{S00}),
it follows that there exists a dense subset $A$ in $L_1(\cM,\tau')$. In particular, 
for   any  $x\in \cF(\tau') \subset \cF(\tau) $ and any $\varepsilon>0$, there exists $z\in A$
such that
\begin{align*}
 \norm{x-z}_\infty \stackrel{\tiny \mbox{\cite[Lemma 25]{DP2014}}}{\le}  \norm{x-z}_{L_1(\cM,\tau')}<\varepsilon.
\end{align*}
Since $\norm{x-z}_{L_1(\cM,\tau)}\le \norm{x-z}_{L_1(\cM,\tau')}  $ for any $z\in \cF(\tau')$, it follows that from \cite[Lemma 25(ii)]{DP2014} there exists an element $z\in A$ such that 
\begin{align}\label{approE}
\norm{x-z}_{E(\cM,\tau)}<\varepsilon. 
\end{align}

Let $x$ be an element in the unit sphere of $ E(\cM,\tau)$ with polar decomposition $x=u|x|$.
Let $\{e_n\}_{n\ge 1}$ be a sequence of mutually orthogonal minimal projections such that $\sum_{n\ge 1}e_n={\bf 1}$. 
Since $\norm{\cdot}_E$ is a Fatou norm, it follows that there exists a finite subset $\{e_n ' \}_{1\le n\le N}$ of $\{e_n\}_{n\ge 1}$ such that
\begin{align}\label{double}
1  = \norm{x}_{E(\cM,\tau)}& = \norm{|x|}_{E(\cM,\tau)}\nonumber  \\
& \ge
\norm{ |x|^{1/2} \cdot \vee_{1\le n\le N} e_{n}'\cdot  |x|^{1/2} }_{E(\cM,\tau)}  \\
& >  \frac{3}{4}\norm{x}_{E(\cM,\tau)} =\frac{3}{4} .\nonumber
\end{align}
Since $N$ is finite and $e_n'$'s are minimal projections, it follows that $ |x|^{1/2} \cdot \vee_{1\le n\le N} e_{n}'\cdot  |x|^{1/2} $ is an element in $\cF(\tau')$ and, therefore in $\cF(\tau)$. 
By \eqref{approE}, there exists $z\in \cA$
such that
\begin{align*}
\norm{2 u   |x|^{1/2} \cdot \vee_{1\le n\le N} e_{n}'\cdot  |x|^{1/2}  -z}_{E(\cM,\tau)} <\frac18.
\end{align*}
This together with \eqref{double} implies that 
\begin{align*}
\frac{11}{8}&<\norm{ 2   |x|^{1/2} \cdot \vee_{1\le n\le N} e_{n}'\cdot  |x|^{1/2}  }_{E(\cM,\tau)} -\frac18 \\
&\le \norm{z}_{E(\cM,\tau)} \\
&\le  \norm{2  |x|^{1/2} \cdot \vee_{1\le n\le N} e_{n}'\cdot  |x|^{1/2}  }_{E(\cM,\tau)}+\frac18 \le\frac{17}{8} . 
\end{align*}
Note that 
$$-  |x| = - |x|^{1/2}  {\bf 1} |x|^{1/2 } \le |x|^{1/2} (2 \cdot  \vee_{1\le n\le N} e_{n}'-{\bf 1})|x|^{1/2} \le |x|^{1/2}  {\bf 1} |x|^{1/2} =  |x|. $$
It follows from \cite{FC88} (see also \cite[Lemma 4.2]{HSZ})
that
$$ |x|^{1/2} (2 \cdot  \vee_{1\le n\le N} e_{n}'-{\bf 1})|x|^{1/2} \prec \prec  |x|. $$
We have $|x|^{1/2} (2\cdot  \vee_{1\le n\le N} e_{n}' -{\bf 1})|x|^{1/2}  \stackrel{\tiny \mbox{\cite[Prop.3(iii)]{DP2014}}}{\in} E(\cM,\tau)$ and
\begin{align*} 
\norm{2u |x|^{1/2}\cdot  \vee_{1\le n\le N} e_{n}' |x|^{1/2}  - x}_E \le \norm{ x}_E =1   . 
\end{align*}
Therefore, we
\begin{align*}
&~\quad \norm{z -x}_{E(\cM,\tau)}\\&\le \norm{ z- 2u |x|^{1/2}\cdot  \vee_{1\le n\le N} e_{n}' |x|^{1/2}}_{E(\cM,\tau)}+ \norm{2u |x|^{1/2}\cdot  \vee_{1\le n\le N} e_{n}' |x|^{1/2} -x}_{E(\cM,\tau)} \\ &<\frac18+\norm{x}_{E(\cM,\tau)}=\frac18+1 =\frac98 <\frac{11}{8}\le \norm{z}_{E(\cM,\tau)}.
\end{align*}
This proves the UBCP of $E(\cM,\tau)$.
\end{proof}

Recall that $\ell_\infty$ has the UBCP\cite{Cheng}. It is also known that  separable symmetric sequence spaces have the UBCP.  The following corollary significantly extends known results on sequence spaces having the UBCP.
\begin{corollary}
Any  symmetric sequence space  $\ell_E(\mathbb{N})$ having the Fatou norm has the UBCP.
\end{corollary}

\begin{remark}
Note that the Fatou norm is essential in Theorems \ref{main} and \ref{main2}.
Indeed, there exists symmetric norm $\norm{\cdot}'$ on $\ell_\infty$ which is not a  Fatou norm such that $(\ell_\infty,\norm{\cdot}')$ fails the BCP
\cite{CCL}.
In the setting of atomless von Neumann algebras, the condition $E(0,1)\ne L_\infty(0,1)$ in Corollary \ref{atomlessnoinfty} can not be removed, see e.g. Theorem~\ref{atomlessnotbcp} and \cite{LZ2020,CKZ}.
\end{remark}

{\bf Acknowledgments} 

 The third author would like to express his gratitude to Lixin Cheng and Chunlan Jiang for inspiring suggestions on ball-covering property and von Neumann algebras.

{\bf Data availability}

Date availability is not applicable to this article as no new date were created or analyzed in this study.

{\bf Conflict of interest}

 The authors have no conflicts of interest to declare that are relevant to the content of this article.

\end{document}